\documentclass{amsart}

\usepackage{amsmath,amssymb,amsthm,upgreek}
\usepackage[colorlinks=true, citecolor=blue,backref=page]{hyperref}
\usepackage{epsfig, graphicx}
\usepackage{verbatim,setspace}

\newcommand{\R}		{\mathbb{R}}
\newcommand{\C}		{\mathbb{C}}
\newcommand{\N}		{\mathbb{N}}
\newcommand{\Z}		{\mathbb{Z}}

\newcommand{\cp}{\mathsf{cp}}
\renewcommand{\deg}{\mathsf{deg}}
\newcommand{\Alg}{\mathcal{S}}
\newcommand{\RS}{\mathfrak{R}}
\newcommand{\bd}{\boldsymbol{\Delta}}
\newcommand{\map}{\Phi}
\newcommand{\mdp}{\mathrm{mod~periods~}}
\newcommand{\z}	{{\boldsymbol z}}
\newcommand{\w}	{{\boldsymbol w}}
\newcommand{\tr}	{{\boldsymbol t}}
\newcommand{\e}	{{\boldsymbol e}}
\newcommand{\ualpha}{\boldsymbol\upalpha}
\newcommand{\ubeta}{\boldsymbol\upbeta}
\newcommand{\ugamma}{\boldsymbol\upgamma}
\newcommand{\rhy}   {\textnormal{RHP}-${\boldsymbol Y}$}
\newcommand{\rhs}   {\textnormal{RHP}-${\boldsymbol X}$}
\newcommand{\rhn}   {\textnormal{RHP}-${\boldsymbol N}$}
\newcommand{\rhr}   {\textnormal{RHP}-${\boldsymbol Z}$}

\newtheorem{theorem}{Theorem}

\newtheorem{lemma}[theorem]{Lemma}
\newtheorem{definition}[theorem]{Definition}

\newtheorem*{st}{Theorem (Stahl)}
\newtheorem*{prt}{Theorem (Perevoznikova-Rakhmanov) \cite{uPerevRakh}}

\begin{document}

\title[Nuttall's theorem with analytic weights on algebraic S-contours]{Nuttall's theorem with analytic weights on algebraic S-contours}

\author[M. Yattselev]{Maxim L. Yattselev}

\address{Department of Mathematical Sciences, Indiana University-Purdue University Indianapolis, 402~North Blackford Street, Indianapolis, IN 46202}

\email{maxyatts@math.iupui.edu}



\dedicatory{Dedicated to the memories of Herbert Stahl, brilliant mathematician and a kind friend, and Andrei~Alexandrovich~Gonchar, great visionary and a wonderful teacher.}

\begin{abstract}

Given a function $f$ holomorphic at infinity, the $n$-th diagonal Pad\'e approximant to $f$, denoted by $[n/n]_f$, is a rational function of type $(n,n)$ that has the highest order of contact with $f$ at infinity. Nuttall's theorem provides an asymptotic formula for the error of approximation $f-[n/n]_f$ in the case where $f$ is the Cauchy integral of a smooth density with respect to the arcsine distribution on $[-1,1]$. In this note, Nuttall's theorem is extended to Cauchy integrals of analytic densities on the so-called algebraic S-contours (in the sense of Nuttall and Stahl).

\end{abstract}

\maketitle

\section{Introduction}
\label{sec:intro}

Let 
\begin{equation}
\label{f}
f(z) = \sum_{k\geq0}f_kz^{-k}
\end{equation}
be a convergent power series. A diagonal Pad\'e approximant to $f$ at infinity is a rational function that has the highest order of contact with $f$ at infinity \cite{Pade92,BakerGravesMorris}. More precisely, let $(P_n,Q_n)$ be a pair of polynomials each of degree at most $n$ satisfying
\begin{equation}
\label{linsys}
R_n(z):= \big(Q_nf-P_n\big)(z) = O\left(1/z^{n+1}\right) \quad \mbox{as} \quad z\to\infty.
\end{equation}
It is not hard to verify that the above relation can be equivalently written as a linear system in terms of the Fourier coefficients of $f$, $P_n$, and $Q_n$ with one more unknown than equations. Therefore the system is always solvable and no solution of it can be such that $Q_n\equiv0$ (we may thus assume that $Q_n$ is monic). In general, a solution of \eqref{linsys} is not unique. However, if $(P_n,Q_n)$ and $(\tilde P_n,\tilde Q_n)$ are two distinct solutions, then $P_n\tilde Q_n-\tilde P_nQ_n\equiv0$ since this difference must behave like $O(1/z)$ near the point at infinity as easily follows from \eqref{linsys}. Thus, each solution of \eqref{linsys} is of the form $(LP_n,LQ_n)$, where $(P_n,Q_n)$ is the unique solution of minimal degree. Hereafter, $(P_n,Q_n)$ will always stand for this unique pair of polynomials. A \emph{diagonal Pad\'e approximant} to $f$ of type $(n,n)$, denoted by $[n/n]_f$, is defined as $[n/n]_f:=P_n/Q_n$. 

We say that a function $f$ of the form \eqref{f} belongs to the class $\Alg$ if it has a meromorphic continuation along any arc originating at infinity that belongs to $\C\setminus E_f$, $\cp(E_f)=0$, and some points in $\C\setminus E_f$ do possess distinct continuations.\footnote{$\cp(\cdot)$ stands for logarithmic capacity \cite{Ransford}.} Given $f\in\Alg$, a compact set $K$ is called \emph{admissible} if $\overline\C\setminus K$ is connected and $f$ has a meromorphic and single-valued extension there. The following theorems summarize one of the fundamental contributions of Herbert Stahl to complex approximation theory \cite{St85,St85b,St86,St97}.

\begin{st}
Given $f\in\Alg$, there exists the unique admissible compact $\Delta_f$ such that $\cp(\Delta_f)\leq\cp(K)$ for any admissible compact $K$ and $\Delta_f\subseteq K$ for any admissible $K$ satisfying $\cp(\Delta_f)=\cp(K)$. Furthermore, Pad\'e approximants $[n/n]_f$ converge to $f$ in logarithmic capacity in $D_f:=\overline\C\setminus\Delta_f$. The domain $D_f$ is optimal in the sense that the convergence does not hold in any other domain $D$ such that $D\setminus D_f\neq\varnothing$.
\end{st}

The minimal capacity set $\Delta_f$, the boundary of the extremal domain $D_f$, has a rather special structure.
\begin{st}
It holds that
\[
\Delta_f = E_0 \cup E_1 \cup \bigcup\Delta_j,
\]
where $E_0\subseteq E_f$, $E_1$ consists of isolated points to which $f$ has unrestricted continuations from the point at infinity leading to at least two distinct function elements, and $\Delta_j$ are open analytic arcs. 
\end{st}

Moreover, the set $\Delta_f$ possesses Stahl's symmetry property.

\begin{st}
It holds that
\[
\frac{\partial g_{\Delta_f}}{\partial \boldsymbol{n}^+} = \frac{\partial g_{\Delta_f}}{\partial \boldsymbol{n}^-} \quad \text{on} \quad \bigcup \Delta_j,
\]
where $\partial/\partial\boldsymbol{n}^\pm$ are the one-sided normal derivatives on $\bigcup \Delta_j$ and $g_{\Delta_f}$ is the Green's function with pole at infinity for $D_f$.
\end{st}

Finally, the arcs $\Delta_j$ can be described as trajectories of a certain quadratic differential.

\begin{st}
Let $h_{\Delta_f}(z)=2\partial_zg_{\Delta_f}(z)$, where $2\partial_z:=\partial_x-\mathrm{i}\partial_y$. The function $h_\Delta^2$ is holomorphic in $D_f$, has a zero of order 2 at infinity, and the arcs $\Delta_k$ are negative critical trajectories of the quadratic differential $h_{\Delta_f}^2(z)\mathrm{d}z^2$. That is, for any smooth parametrization $z(t):(0,1)\to\Delta_j$ it holds that $h_{\Delta_f}^2(z(t))\big(z^\prime(t)\big)^2< 0$ for all $t\in(0,1)$.
\end{st}

If now $f\in\mathcal{S}$ is an algebraic function, then the set $E_f$ is finite and so is the collection $\bigcup\Delta_j$. This motivated the following definition.

\begin{definition}
\label{SContour}
A compact set $\Delta$ is called an algebraic S-contour if the complement of $\Delta$, say $D$, is connected,
\[
\Delta=E_0\cup E_1\cup\bigcup \Delta_j,
\]
where $\bigcup \Delta_j$ is a finite union of open analytic arcs, $E_0\cup E_1$ is a finite set of points such that each element of $E_0$ is an endpoint of exactly one arc $\Delta_j$ while each element of $E_1$ is an endpoint of at least three arcs, and
\[
\frac{\partial g_\Delta}{\partial\boldsymbol{n}^+} = \frac{\partial g_\Delta}{\partial\boldsymbol{n}^-} \quad \mbox{on} \quad \bigcup \Delta_j,
\]
where $g_\Delta$ is the Green's function for $D$ with pole at infinity.
\end{definition}

Any algebraic S-contour is a minimal capacity contour for some algebraic function $f$. Given $\Delta$, an eligible function $f_\Delta\in\Alg$ can be constructed in the following way. Denote by $m$ the number of connected components of $\Delta$, by $E_{0j}$ the intersection of $E_0$ with the $j$-th connected component, and by $m_j$ the cardinality of $E_{0j}$. Then one can take $f_\Delta(z)=\sum_{j=1}^m\big(\prod_{e\in E_{0j}}(z-e)\big)^{-1/m_j}$.

Algebraic S-contours admit a description via critical trajectories of \emph{rational} quadratic differentials.  For such a contour $\Delta$, set
\begin{equation}
\label{h}
h_\Delta(z) := 2\partial_z g_\Delta(z).
\end{equation}
For each point $e\in E_0\cup E_1$ denote by $i(e)$ the \emph{bifurcation index} of $e$, that is, the number of different arcs $\Delta_j$ incident with $e$. It follows immediately from the definition of an algebraic S-contour that $i(e)=1$ for $e\in E_0$ and $i(e)\geq3$ for $e\in E_1$. Denote also by $E_2$ the set of \emph{critical points} of $g_\Delta$ with $j(e)$ standing for the \emph{order} of $e\in E_2$, i.e., $\partial_z^jg_D(e)=0$ for $j\in\{1,\ldots,j(e)\}$ and $\partial_z^{j(e)+1}g_D(e)\neq0$. The set $E_2$ is necessarily finite.

\begin{prt}
Let $\Delta$ be an algebraic S-contour. Then the arcs $\Delta_k$ are negative critical trajectories of the quadratic differential $h_\Delta^2(z)\mathrm{d}z^2$. Moreover,
\[
h_\Delta^2(z) = \prod_{e\in E_0\cup E_1}(z-e)^{i(e)-2}\prod_{e\in E_2}(z-e)^{2j(e)}
\]
and $h_\Delta^2(z)=z^{-2}+O\big(z^{-3}\big)$ as $z\to\infty$.
\end{prt}

The reason to restrict our attention from all possible S-contours to the algebraic ones is that one might hope for a stronger convergence than convergence in capacity. Indeed, it was suggested by Nuttall \cite{Nut84} that if
\begin{equation}
\label{f_rho}
f_\rho(z) := \frac1{2\pi\mathrm{i}}\int_\Delta\frac{(\rho/w_\Delta^+)(t)}{t-z}\mathrm{d}t, \quad z\in \overline\C\setminus\Delta,
\end{equation}
where $\rho$ is a H\"older continuous and non-vanishing  function on an algebraic S-contour $\Delta$ and
\begin{equation}
\label{w_Delta}
w_\Delta^2(z) := \prod_{e\in E_\Delta}(z-e)
\end{equation}
with $E_\Delta$ being the subset of $E_0\cup E_1$ consisting of points having odd bifurcation index, then the diagonal Pad\'e approximants $[n/n]_{f_\rho}$ converge to $f_\rho$ \emph{``nearly''} uniformly in $D$ (uniformly if $\Delta$ is an interval). The absence of the uniform convergence is due to the presence of a finite number of \emph{``wandering''} or \emph{``spurious''} poles \cite{Gon82,St98}, see the discussion after Theorem~\ref{thm:main} further below. The presence of these poles was already observed by Akhiezer \cite[Section~53]{Akhiezer} and \cite{Akh60}, who considered the case of $\Delta$ being a union of several real intervals and $\rho$ being a positive polynomial on $\Delta$ (the so-called Bernstein-Szeg\H{o} case). Nuttall himself, in the joint work with Singh \cite{NutS77}, extended Akhiezer's method to an arbitrary algebraic S-contour and an arbitrary non-vanishing polynomial (getting rid of positivity). Later, Nuttall showed the validity of his claim on an interval \cite{Nut90} using the method of the singular integral equations. With the help of this method, Nuttall's claim has been verified by Suetin \cite{Suet00,Suet03} when $\Delta$ is a disjoint union of analytic arcs and by Baratchart and the author \cite{BY13} when $\Delta$ is a union of three arcs meeting at one point. Mart\'inez Finkelshtein, Rakhmanov, and Suetin also considered the case of connected algebraic S-contours and semi-classical weights using WKB analysis \cite{M-FRakhSuet12}. In this note we prove Nuttall's theorem on an arbitrary algebraic S-contour but only when $\rho$ in \eqref{f_rho} is holomorphic and non-vanishing in a neighborhood of $\Delta$.  The proof of the full Nuttall's theorem will appear elsewhere \cite{uY3}. 

This note is complimentary to \cite{uAY} by Aptekarev and the author, where the same problem is considered but it is only required that $\rho$ is holomorphic across each $\Delta_j$ and can vanish or blow up at the points of $E_0\cup E_1$. However, \cite{uAY} places the restriction on the algebraic S-contours requiring the bifurcation index $i(e)$ to be either 1 or 3 (no such restriction is placed here). This note as well as \cite{uAY} use the matrix Riemann-Hilbert approach that requires local analysis around the points in $E_0\cup E_1$ unless the weight is precisely as in \eqref{f_rho} with $\rho$ non-vanishing and holomorphic (this was first observed by Aptekarev and Van Assche for the case of an interval \cite{ApVA04}). This is the reason for the difference in assumptions between \cite{uAY} and this note.

This paper is organized as follows. In the next section we construct the Riemann surface of $h_\Delta$, which turns out to be the ``correct'' domain of definition for the functions describing the asymptotics of Pad\'e approximants. The latter functions are then introduced as solutions to a certain family of boundary value problems on the constructed surface. With these preliminaries out of the way, we prove the main result in last section using the matrix Riemann-Hilbert analysis.

\section{Boundary Value Problem}
\label{sec:bvp}

Fix an algebraic S-contour $\Delta$ with complement $D$ and let $h_\Delta$ be given by \eqref{h}.

\subsection{Riemann Surface} Denote by $\RS$ the Riemann surface defined by $h_\Delta$ or equivalently by $w_\Delta$. We represent $\RS$ as a two-sheeted ramified cover of $\overline\C$ constructed in the following manner. Two copies of $\overline\C$ are cut along each arc $\Delta_j$. These copies are clipped together at the elements of $E_\Delta\subseteq E_0\cup E_1$ (branch points of $h_\Delta$).  These copies are further glued together along the cuts in such a manner that the right (resp. left) side of the arc $\Delta_j$ belonging to the first copy, say $\RS^{(0)}$, is joined with the left (resp. right) side of the same arc $\Delta_j$ only belonging to the second copy, $\RS^{(1)}$. The genus of $\RS$, which we denote by $g$, satisfies the equality $2(g+1)=|E_\Delta|$.

According to the above construction, each arc $\Delta_j$ together with its endpoints corresponds to a cycle, say $\bd_j$, on $\RS$. We set $\bd:=\bigcup_j\bd_j$, denote by $\pi$ the canonical projection $\pi:\RS\to\overline\C$, and define
\[
D^{(k)}:=\RS^{(k)}\cap \pi^{-1}(D) \quad \mbox{and} \quad z^{(k)}:=D^{(k)}\cap\pi^{-1}(z)
\]
for $k\in\{0,1\}$ and $z\in D$. We further set $\boldsymbol E_\Delta:=\pi^{-1}(E_\Delta)$, which is comprised exactly of the ramification points of $\RS$. The cycles $\bd_j$ are oriented so that $D^{(0)}$ remains on the left when $\bd_j$ is traversed in the positive direction. We designate the symbol $\cdot^*$ to stand for the conformal involution acting on the points of $\RS$ that fixes the ramification points $\boldsymbol E_\Delta$ and sends $z^{(k)}$ into $z^{(1-k)}$, $k\in\{0,1\}$. We use bolds lower case letters such as $\z,\tr$ to indicate points on $\RS$ with canonical projections $z,t$.

Since $h_\Delta$ has only square root branching, each connected component of $\Delta$ contains even number of branch points. This allows us to number these points, $E_\Delta=\{e_0,e_1,\ldots,e_{2g+1}\}$, in the following fashion. If we consider $\partial D$ as a positively oriented Jordan curve (this way it contains two copies of each $\Delta_j$) and traverse it in the positive direction starting at $e_{2k}$, the next encountered branch point should be $e_{2k+1}$, $k\in\{1,\ldots,g\}$.

Denote by $\ualpha_k$, $k\in\{1,\ldots,g\}$, a smooth involution-symmetric, i.e.,  $\ualpha_k=\{\z^*|\z\in\ualpha_k\}$, Jordan curve that passes through $\e_1$ and $\e_{2k}$, and no other point of $\bd$ (until the end of the subsection we assume that $g\geq1$), which is oriented so that the positive direction in $D^{(0)}$ goes from $\e_1$ to $\e_{2k}$. We require that $\ualpha_k\cap\ualpha_j=\{\e_1\}$ for each pair $k\neq j$. We further denote by $\ubeta_k$ a smooth involution-symmetric Jordan curve that passes through $\e_{2k}$ and $\e_{2k+1}$ and is oriented so that at the point of intersection the tangent vectors to $\ualpha_k,\ubeta_k$ form the right pair. Again, we suppose that $\bd\cap\ubeta_k=\{\e_{2k},\e_{2k+1}\}$ and also assume that $\ubeta_j$ has empty intersection with any cycle $\ugamma\in\big\{\ualpha_k,\ubeta_k\big\}_{k=1}^g$ except for $\ualpha_j$ with which it has only one point in common, necessarily $\e_{2j}$. Set
\[
\widetilde\RS:=\RS\setminus\bigcup_{k=1}^g(\ualpha_k\cup\ubeta_k) \quad \mbox{and} \quad \widehat\RS:=\RS\setminus\bigcup_{k=1}^g\ualpha_k.
\]
The constructed collection $\big\{\ualpha_k,\ubeta_k\big\}_{k=1}^g$ forms a homology basis on $\RS$ and so defined $\widetilde\RS$ is simply connected. In the case $g=0$ these definitions are void and the whole surface is conformally equivalent to the Riemann sphere $\overline\C$.

\subsection{Differentials on $\RS$}

Denote by $\mathrm{d}\vec\Omega:=\left(\mathrm{d}\Omega_1,\ldots,\mathrm{d}\Omega_g\right)^\mathsf{T}$ the column vector of $g$ linearly independent holomorphic differentials normalized so that $\oint_{\ualpha_k}\mathrm{d}\vec\Omega = \vec e_k$, $k\in\{1,\ldots,g\}$, where $\left\{\vec e_k\right\}_{k=1}^g$ is the standard basis for $\R^g$ and $\vec e^\mathsf{T}$ is the transpose of $\vec e$. Since the genus of $\RS$ is $g$, the differentials $\mathrm{d}\Omega_k$ form a basis for the space of holomorphic differentials on $\RS$. Set
\begin{equation}
\label{B}
\mathbf{B} := \left[ \oint_{\ubeta_j}\mathrm{d}\Omega_k\right]_{j,k=1}^g.
\end{equation}
It is known that the matrix $\mathbf{B}$ is symmetric and has positive definite imaginary part. Set
\begin{equation}
\label{wR}
w\big(z^{(k)}\big) := (-1)^kw_\Delta(z), \quad z\in D,
\end{equation}
which is continuous across $\bd$ and therefore is rational on $\RS$. It can be argued that
\begin{equation}
\label{ExplicitOmegas}
\mathrm{d}\Omega_j(\z) = (L_j/w)(\z)\mathrm{d}z,
\end{equation}
for some $L_j$, which is a polynomial in $z$ lifted to $\RS$ of degree at most $g-1$.

Analogously to \eqref{wR}, the function 
\begin{equation}
\label{hR}
h\big(z^{(k)}\big) := (-1)^kh_\Delta(z), \quad z\in D,
\end{equation}
extends to $\bd$ by continuity and is rational on $\RS$. By setting $\mathrm{d}G(\z)=h(\z)\mathrm{d}z$, we obtain the so-called \emph{Green's differential} on $\RS$. That is, all the periods (integrals over cycles on $\RS$) of $\mathrm{d}G$ are purely imaginary and $\mathrm{d}G$ is meromorphic having two simple poles at $\infty^{(1)}$ and $\infty^{(0)}$ with respective residues 1 and $-1$ (it holds that $\mathrm{d}G(z^{(k)})=((-1)^{k+1}/\zeta+\textnormal{holomorphic})\mathrm{d}\zeta$ in local coordinates $\zeta=1/z^{(k)}$). Thus, we can define two vectors of real constants  $\vec\omega=(\omega_1,\ldots,\omega_g)^\mathsf{T}$ and $\vec\tau=(\tau_1,\ldots,\tau_g)^\mathsf{T}$ by
\begin{equation}
\label{constants}
\omega_k:=-\frac1{2\pi\mathrm{i}}\oint_{\ubeta_k}\mathrm{d}G \quad \mbox{and} \quad \tau_k:=\frac1{2\pi\mathrm{i}}\oint_{\ualpha_k}\mathrm{d}G.
\end{equation}

\subsection{Mapping Function}
Define
\begin{equation}
\label{phi}
\map(\z) := \exp\left\{\int_{\e_0}^\z \mathrm{d}G\right\}, \quad \z\in\widetilde\RS.
\end{equation}
The function $\map$ is holomorphic and non-vanishing on $\widetilde\RS$ except for a simple pole at $\infty^{(0)}$ and a simple zero at $\infty^{(1)}$. Furthermore, it possesses continuous traces on both sides of each cycle of the canonical basis that satisfy
\begin{equation}
\label{jumpa}
\map^+ = \map^-\left\{
\begin{array}{ll}
\displaystyle \exp\big\{2\pi \mathrm{i}\omega_k\big\} & \mbox{on} \quad \ualpha_k, \\
\displaystyle \exp\big\{2\pi \mathrm{i}\tau_k\big\} & \mbox{on} \quad \ubeta_k.
\end{array}
\right.
\end{equation}
In the case $g=0$, $\Phi$ is a rational function well-defined on the whole Riemann surface.

Observe that the path of integration in \eqref{hR} always can be chosen so it completely belongs to either $\RS^{(0)}$ or $\RS^{(1)}$. Thus, it readily follows from \eqref{hR} and \eqref{h} that
\begin{equation}
\label{IneqPhi}
\Phi(z^{(k)}) = \exp\left\{(-1)^k\int_{e_0}^zh_\Delta(t)\mathrm{d}t\right\} \quad \mbox{and} \quad \big|\Phi(z^{(k)})\big|=\exp\left\{(-1)^kg_D(z)\right\}
\end{equation}
for $z\in D$. This computation has a trivial but remarkably important consequence, namely,
\begin{equation}
\label{ineq1}
\Phi(z^{(0)})\Phi(z^{(1)})\equiv1 \quad \mbox{and} \quad |\Phi(z^{(0)})|>|\Phi(z^{(1)})|, \quad z\in D.
\end{equation}
When $g=0$, the pull back of $\Phi$ from $D^{(0)}$ to $D$ is nothing else but the conformal map of $D$ onto $\{|z|>1\}$ fixing the point at infinity and sending $e_0$ to 1.

\subsection{Cauchy Kernel}
Let $\ugamma$ be an involution-symmetric, piecewise-smooth oriented chain on $\RS$ that has only finitely many points in common with the $\ualpha$-cycles. Further, let $\lambda$ be a H\"older continuous function on $\ugamma$. That is, for each $\z\in\ugamma$, $\lambda\circ\phi_{\z}$ is H\"older continuous on $\phi_\z^{-1}(\ugamma)$ where $\phi_\z$ is a holomorphic local parametrization around $\z$. 

Denote by $\mathrm{d}\Omega_{\z,\z^*}$ the normalized abelian differential of the third kind (i.e., it is a meromorphic differential with two simple poles at $\z$ and $\z^*$ with respective residues 1 and $-1$ normalized to have zero periods on the $\ualpha$-cycles). Set
\[
\Lambda(\z) := \frac1{4\pi\mathrm{i}}\oint_{\ugamma}\lambda\mathrm{d}\Omega_{\z,\z^*}, \quad \z\not\in\ugamma.
\]
It is known \cite[Eq.~(2.7)--(2.9)]{Zver71} that $\Lambda$ is a holomorphic function in $\widehat\RS\setminus\ugamma$, $\Lambda(\z)+\Lambda(\z^*)\equiv0$ there, the traces $\Lambda^\pm$ are continuous and satisfy
\[
\Lambda^+(\z) - \Lambda^-(\z) = \frac12\left\{
\begin{array}{rl}
\displaystyle \lambda(\z)+\lambda(\z^*), & \z\in\ugamma, \\
\displaystyle -2\oint_{\ugamma} \lambda\mathrm{d}\Omega_k, & \z\in\ualpha_k\setminus\ugamma.
\end{array}
\right.
\]
That is, the differential $\mathrm{d}\Omega_{\z,\z^*}$ is a discontinuous Cauchy kernel on $\RS$ (it is discontinuous as $\Lambda$ has additional jumps across the $\ualpha$-cycles).

\subsection{Auxiliary Functions, I}
To remove the jumps of $\Phi$ across the $\ubeta$-cycles, define $\lambda_{\vec \tau}$ to be the function on $\ugamma=\cup\ubeta_k$ such that $\lambda_{\vec \tau}\equiv-2\pi\mathrm{i}\tau_k$ on $\ubeta_k$ and set
\begin{equation}
\label{stau}
S_{\vec \tau}(\z) := \exp\big\{\Lambda_{\vec\tau}(\z)\big\}, \quad \z\in\widetilde\RS.
\end{equation}
Then $S_{\vec \tau}$ is a holomorphic function in $\widetilde\RS$ with continuous traces that satisfy
\begin{equation}
\label{staujumps}
S_{\vec \tau}^+ = S_{\vec \tau}^-\left\{
\begin{array}{ll}
\displaystyle \exp\big\{2\pi \mathrm{i}\big(\mathbf B\vec\tau\:\big)_k\big\} & \mbox{on} \quad \ualpha_k, \smallskip \\
\displaystyle \exp\big\{-2\pi \mathrm{i}\tau_k\big\} & \mbox{on} \quad \ubeta_k,
\end{array}
\right.
\end{equation}
where the upper equality follows straight from \eqref{B} and we adopt the convention $(\vec c)_k=c_k$ for $\vec c=(c_1,\ldots,c_g)$.

Let now $\rho$ be a non-vanishing holomorphic function on $\Delta$. As $\rho$ is non-vanishing, one can select a smooth branch of $\log\rho$, which we lift to $\bd$, $\lambda_\rho:=-\log\rho\circ\pi$. Define
\begin{equation}
\label{Sp}
S_\rho(\z) := \exp\big\{\Lambda_\rho(\z)\big\}, \qquad \vec c_\rho := -\frac1{2\pi\mathrm{i}}\oint_{\bd}\lambda_\rho\mathrm{d}\vec\Omega.
\end{equation}
Then $S_\rho$ is a holomorphic and non-vanishing function in $\widehat\RS\setminus\bd$ with continuous traces that satisfy
\begin{equation}
\label{jumpSp}
S_\rho^+ = S_\rho^-\left\{
\begin{array}{ll}
\displaystyle \exp\big\{2\pi \mathrm{i}\big(\vec c_\rho\big)_k\big\} & \mbox{on} \quad \ualpha_k, \\
1/\rho\circ\pi & \mbox{on} \quad \bd.
\end{array}
\right.
\end{equation}

By gathering together \eqref{jumpa}, \eqref{staujumps}, \eqref{jumpSp} and setting $S_{n\vec\tau}:=S^n_{\vec\tau}$, we deduce that
\begin{equation}
\label{jump1}
(\Phi^n S_\rho S_{n\vec\tau})^+ = (\Phi^n S_\rho S_{n\vec\tau})^-\left\{
\begin{array}{ll}
\displaystyle \exp\big\{2\pi \mathrm{i}\big(\vec c_\rho+n\big(\vec\omega+\mathbf B\vec\tau\big)\big)_k\big\} & \mbox{on} \quad \ualpha_k, \\
\displaystyle 1/\rho\circ\pi & \mbox{on} \quad \bd.
\end{array}
\right.
\end{equation}

\subsection{Jacobi Inversion Problem}
To remove the jump of $\Phi^n S_\rho S_{n\vec\tau}$ from the $\ualpha$-cycles, let us digress into explaining what a Jacobi inversion problem is.

An \emph{integral divisor} is a formal symbol of the form $\mathcal{D}=\sum n_j\z_j$, where $\{\z_j\}$ is an arbitrary finite collection of distinct points on $\RS$ and $\{n_j\}$ is a collection of positive integers. The sum $\sum n_j$ is called the \emph{degree} of the divisor $\mathcal{D}$. Let $\mathcal{D}_1=\sum n_j\z_j$ and $\mathcal{D}_2=\sum m_j\w_j$ be integral divisors. A divisor $\mathcal{D}_1-\mathcal{D}_2$ is called \emph{principal} if there exists a rational function on $\RS$ that has a zero at every $\z_j$ of multiplicity $n_j$,  a pole at every $\w_j$ of order $m_j$, and otherwise is non-vanishing and finite. By Abel's theorem, $\mathcal{D}_1-\mathcal{D}_2$ is principal if and only if the divisors $\mathcal{D}_1$ and $\mathcal{D}_2$ have the same degree and
\[
\vec\Omega(\mathcal{D}_1) - \vec\Omega(\mathcal{D}_2) \equiv \vec 0 \quad \big(\mdp\mathrm{d}\vec\Omega\big),
\]
where $\vec\Omega(\mathcal{D}_1) := \sum n_j\int_{\e_0}^{\z_j}\mathrm{d}\vec\Omega$ and the equivalence of two vectors $\vec c,\vec e\in \C^g$ is defined by $\vec c\equiv \vec e$ $\big(\mdp\mathrm{d}\vec\Omega\big)$ if and only if $\vec c - \vec e = \vec j + \mathbf{B}\vec m$ for some $\vec j,\vec m\in\Z^g$.

Set $\mathcal{D}_*=g\infty^{(1)}$. We are seeking a solution of the following Jacobi inversion problem: find an integral divisor $\mathcal{D}$ of degree $g$ such that
\begin{equation}
\label{main-jip}
\vec\Omega(\mathcal{D})-\vec\Omega(\mathcal{D}_*) \, \equiv \, \vec c_\rho + n\big(\vec\omega+\mathbf{B}\vec\tau\:\big) \quad \left(\mdp \mathrm{d}\vec\Omega\right),
\end{equation}
where the vectors $\vec\omega$ and $\vec\tau$ were defined in \eqref{constants}. This problem is always solvable and the solution is unique up to a principal divisor. That is, if $\mathcal{D} - \big\{\mbox{ principal divisor }\big\}$ is an integral divisor, then it also solves \eqref{main-jip}. Immediately one can see that the subtracted principal divisor should have an integral part of degree at most $g$. As $\RS$ is hyperelliptic, such divisors come solely from rational functions on $\overline\C$ lifted to $\RS$. In particular, such principal divisors are involution-symmetric. Hence, if a solution of \eqref{main-jip} contains at least one pair of involution-symmetric points, then replacing this pair by another such pair produces a different solution of \eqref{main-jip}. However, if a solution does not contain such a pair, then it solves \eqref{main-jip} uniquely. 

\subsection{Solutions of the JIP}
In what follows, we denote by $\mathcal{D}_n$ either the unique solution of \eqref{main-jip} or the solution where each conjugate-symmetric pair is replaced by $\infty^{(0)}+\infty^{(1)}$. We further set $\N_*$ to be the subsequence of all indices for which \eqref{main-jip} is uniquely solvable. Non-unique solutions are related to unique solutions in the following manner:
\begin{equation}
\label{Dn}
\mathcal{D}_n=\sum_{i=1}^{g-l}\tr_i+k\infty^{(0)}+(l-k)\infty^{(1)} \quad \Leftrightarrow \quad \mathcal{D}_{n+j} = \mathcal{D}_n +j\big(\infty^{(0)}-\infty^{(1)}\big),
\end{equation}
for $j\in\{-k,\ldots,l-k\}$, where $l>0$, $k\in\{0,\ldots,l\}$, and $|t_i|<\infty$. Indeed, Riemann's relations state that
\[
\oint_{\ubeta_k}\mathrm{d}\Omega_{\infty^{(1)},\infty^{(0)}} = 2\pi\mathrm{i}\int_{\infty^{(0)}}^{\infty^{(1)}}\mathrm{d}\Omega_k
\]
for each $k\in\{1,\ldots,g\}$, where the path of integration lies entirely in $\widetilde\RS$. Since the differentials $\mathrm{d}\Omega_{\infty^{(1)},\infty^{(0)}}$ and $\mathrm{d}G$ have the same poles with the same residues, they differ by a holomorphic differential. Their normalizations imply that
\[
\mathrm{d}G = \mathrm{d}\Omega_{\infty^{(1)},\infty^{(0)}} + 2\pi\mathrm{i}\sum_{k=1}^g\tau_k\mathrm{d}\Omega_k.
\]
Combining the last two equations with \eqref{B} and \eqref{constants} we get that
\begin{equation}
\label{diffOmegas}
\vec\Omega\big(\infty^{(0)}\big) - \vec\Omega\big(\infty^{(1)}\big) = \vec\omega+\mathbf{B}\vec\tau,
\end{equation}
which immediately implies that
\[
\vec \Omega(\mathcal{D}_n) - \vec \Omega(\mathcal{D}_*)+ j\left(\vec\Omega\big(\infty^{(0)}\big) - \vec\Omega\big(\infty^{(1)}\big)\right) \equiv \vec c_\rho + (n+j)\big(\vec\omega+\mathbf{B}\vec\tau\:\big)
\]
from which \eqref{Dn} easily follows. In particular, \eqref{Dn} implies the unique solvability of \eqref{main-jip} for the indices $n-k$ and $n+l-k$. 

In another connection, if $\mathcal{D}_n$ is a unique solution of \eqref{main-jip} that does not contain $\infty^{(k)}$, $k\in\{0,1\}$, then $\mathcal{D}_{n-(-1)^k}$ is also a unique solution of \eqref{main-jip} as otherwise it would contain at least one pair $\infty^{(1)}+\infty^{(0)}$, which would imply that $\mathcal{D}_n$ contains $\infty^{(k)}$ by \eqref{Dn}. Moreover, the divisors $\mathcal{D}_n$ and $\mathcal{D}_{n-(-1)^k}$ have no points in common. Indeed, denote by $\mathcal{D}$ the common part. Then
\begin{equation}
\label{previous-next}
\vec \Omega(\mathcal{D}_n) - \vec \Omega\big(\mathcal{D}_{n-(-1)^k}\big) - (-1)^k\left(\vec\Omega\big(\infty^{(0)}\big) - \vec\Omega\big(\infty^{(1)}\big)\right) \equiv \vec 0 \quad (\mdp \mathrm{d}\vec\Omega)
\end{equation}
and therefore the divisor $\mathcal{D}_n-\mathcal{D}_{n-(-1)^k}-(-1)^k \infty^{(0)} +(-1)^k \infty^{(1)}$ is principal. However, if the degree of $\mathcal{D}$ were strictly positive, the integral part of the constructed divisor would be at most $g$. Such divisors come solely from rational functions on $\overline\C$ lifted to $\RS$ and are involution-symmetric. Hence, the divisor $\mathcal{D}_n-\mathcal{D}$ would contain an involution-symmetric pair or $\infty^{(k)}$. As both conclusions are impossible, the claim indeed takes place.

\subsection{Limit Points}
\label{ssec:limitpoints}

One can consider integral divisors of degree $g$ as elements of $\RS^g/\Sigma_g$, the quotient of $\RS^g$ by the symmetric group $\Sigma_g$, which is a compact topological space. Thus, it make sense to talk about the limit points of $\{\mathcal{D}_n\}$. The considerations of the previous section extend to them in the following manner.

Let $\N^\prime\subseteq\N$ be such that $\mathcal{D}_n\to \mathcal{D}^\prime$, $n\in\N^\prime$, for some divisor $\mathcal{D}^\prime$. In the most general form the divisor $\mathcal{D}^\prime$ can be written as
\[
\mathcal{D}^\prime = \mathcal{D} + \sum_{i=1}^k\left(z_i^{(0)}+z_i^{(1)}\right) + l_0\infty^{(0)} + l_1\infty^{(1)},
\]
where the integral divisor $\mathcal{D}$ has degree $g-2k-l_0-l_1$, is non-special, and does not contain neither $\infty^{(0)}$ nor $\infty^{(1)}$. Let further $\N^{\prime\prime}\subseteq\N^\prime$ be another subsequence such that the divisors $\mathcal{D}_{n+l_1+k}$, $n\in\N^{\prime\prime}$, converge to some divisor, say $\mathcal{D}^{\prime\prime}$. Then the continuity of $\vec\Omega$ implies that
\[
\lim_{\N^{\prime\prime}\ni n\to\infty} \vec\Omega\big(\mathcal{D}_n\big) = \vec\Omega\big(\mathcal{D}^{\prime}\big) \quad \text{and} \quad \lim_{\N^{\prime\prime}\ni n\to\infty} \vec\Omega\big(\mathcal{D}_{n+l_1+k}\big) = \vec\Omega\big(\mathcal{D}^{\prime\prime}\big)
\]
with all the paths of integration belonging to $\widetilde\RS$. That is,
\[
\left\{
\begin{array}{lll}
\displaystyle  \lim_{\N^{\prime\prime}\ni n\to\infty} \left(\vec c_\rho+n\big(\vec\omega+\mathbf{B}\vec\tau\big)\right) &\equiv& \displaystyle \vec\Omega\big(\mathcal{D}^{\prime}\big) - \vec\Omega\big(\mathcal{D}_*\big), \smallskip \\
\displaystyle \lim_{\N^{\prime\prime}\ni n\to\infty} \left(\vec c_\rho+(n+l_1+k)\big(\vec\omega+\mathbf{B}\vec\tau\big)\right) &\equiv& \displaystyle \vec\Omega\big(\mathcal{D}^{\prime\prime}\big) - \vec\Omega\big(\mathcal{D}_*\big),
\end{array}
\right.
\]
since $\mathcal{D}_n$ solves \eqref{main-jip}. Hence, it holds by \eqref{diffOmegas} that
\[
\vec\Omega\big(\mathcal{D}^{\prime\prime}\big) \equiv \vec\Omega\big(\mathcal{D}^{\prime}\big) + (l_1+k)\left( \vec\Omega\big(\infty^{(0)}\big) - \vec\Omega\big(\infty^{(1)}\big)  \right).
\]
Observe also that $\vec\Omega\big(z^{(0)}\big)=-\vec\Omega\big(z^{(1)}\big)$ as follows from \eqref{wR} and \eqref{ExplicitOmegas}. Thus, the above congruence can be rewritten as
\[
\vec\Omega\big(\mathcal{D}^{\prime\prime}\big) \equiv \vec\Omega\big(\mathcal{D}\big) + (l_0+l_1+2k)\vec\Omega\left(\infty^{(0)}\right).
\]
Therefore, it follows from Abel's theorem that the divisor $\mathcal{D} + (l_0+l_1+2k)\infty^{(0)} - \mathcal{D}^{\prime\prime}$ is principal. However, it is also special and does not contain any involution-symmetric pairs, which is possible only if it is identically zero. That is,
\[
\mathcal{D}^{\prime\prime} = \mathcal{D} + (l_0+l_1+2k)\infty^{(0)}.
\]
In fact, exactly as in the preceding subsection, we could take the second sequence to be $\mathcal{D}_{n+j}$ for any $j\in\{-l_0-k,\ldots,l_1+k\}$ and arrive at similar conclusions, see \cite[Proposition~2]{uAY}.

Moreover, let now $\N^{\prime\prime\prime}\subseteq \N^{\prime\prime}$ be such that $\mathcal{D}_{n+l_1+k+1}\to\mathcal{D}^{\prime\prime\prime}$ for some divisor $\mathcal{D}^{\prime\prime\prime}$. It follows from the considerations as above and the argument used in \eqref{previous-next} applied to $\mathcal{D}^{\prime\prime\prime}$ and $\mathcal{D}^{\prime\prime}$ that $\mathcal{D}^{\prime\prime\prime}$ is non-special and disjoined from $\mathcal{D}^{\prime\prime}$.
\subsection{Riemann's Theta Function}

The solution of the Jacobi inversion problem \eqref{main-jip} helps to remove the jump from the $\ualpha$-cycles in \eqref{jump1} via \emph{Riemann's theta function}. The theta function associated with $\mathbf B$ is an entire transcendental function of $g$ complex variables defined by
\[
\theta\left(\vec u\right) := \sum_{\vec n\in\Z^g}\exp\bigg\{\pi\mathrm{i}\vec n^\mathsf{T}\mathbf B\vec n + 2\pi\mathrm{i}\vec n^\mathsf{T}\vec u\bigg\}, \quad \vec u\in\C^g.
\]
As shown by Riemann, the symmetry of $\mathbf B$ and positive definiteness of its imaginary part ensures the convergence of the series for any $\vec u$. It can be directly checked that $\theta$ enjoys the following periodicity properties:
\begin{equation}
\label{theta-periods}
\theta\left(\vec u + \vec j + \mathbf B\vec m\right) = \exp\bigg\{-\pi \mathrm{i}\vec m^\mathsf{T}\mathbf B \vec m - 2\pi \mathrm{i}\vec m^\mathsf{T}\vec u\bigg\}\theta\big(\vec u\big), \quad \vec j,\vec m\in\Z^g.
\end{equation}

Specializing integral divisors to one point $\z$, we reduce $\vec\Omega(\z)$ to a vector of holomorphic functions in $\widetilde\RS$ with continuous traces on the cycles of the homology basis that satisfy
\begin{equation}
\label{Omega-jump}
\vec\Omega^{+}-\vec\Omega^- = \left\{
\begin{array}{rl}
-\mathbf B\vec e_k & \mbox{on} \quad \ualpha_k,\\
\vec e_k & \mbox{on} \quad \ubeta_k,
\end{array}
\right.
\end{equation}
$k\in\{1,\ldots,g\}$. It readily follows from the relations above that each $\Omega_k$ is, in fact, holomorphic in $\widehat\RS\setminus\ubeta_k$. It is known that
\[
\theta\left(\vec u\right)=0 \quad \Leftrightarrow \quad \vec u\equiv \vec\Omega\left(\mathcal{D}_{\vec u}\right) + \vec K \quad \left(\mdp d\vec\Omega\right)
\]
for some integral divisor $\mathcal{D}_{\vec u}$ of degree $g-1$, where $\vec K$ is the vector of Riemann constants defined by $(\vec K)_j:=([\mathbf B]_{jj}-1)/2-\sum_{k\neq j}\oint_{\ualpha_k}\Omega_j^-\mathrm{d}\Omega_k$, $j\in\{1,\ldots,g\}$.

For $n\in\N_*$ ($\mathcal{D}_n$ is unique, and hence does not contain involution-symmetric pairs), set
\begin{equation}
\label{thetan}
\Theta_n(\z) := \frac{\theta\left(\vec\Omega(\z) - \vec\Omega(\mathcal{D}_n) - \vec K\right)}{\theta\left(\vec\Omega(\z) - \vec\Omega(\mathcal{D}_*) - \vec K\right)}.
\end{equation}
Since the divisors $\mathcal{D}_n$ and $\mathcal{D}_*$ do not contain involution-symmetric pairs, $\vec\Omega(\z)+\vec\Omega(\z^*)\equiv0$, and $\theta(-\vec u)=\theta(\vec u)$,  $\Theta_n$ is a multiplicatively multi-valued meromorphic function on $\RS$ with zeros at the points of the divisor $\mathcal{D}_n$ of respective multiplicities, a pole of order $g$ at $\infty^{(1)}$, and otherwise non-vanishing and finite (there will be a reduction of the order of the pole at $\infty^{(1)}$ when the divisor $\mathcal{D}_n$ contains this point). In fact, it is meromorphic and single-valued in $\widehat\RS$ and
\begin{eqnarray}
\Theta_n^+  &=& \Theta_n^- \exp\left\{2\pi \mathrm{i}\big(\Omega_k(\mathcal{D}_*)-\Omega_k(\mathcal{D}_n)\big)\right\} \nonumber \\
\label{jump2}
&=& \Theta_n^- \exp\left\{-2\pi \mathrm{i}\left(\vec c_\rho + n\big(\vec\omega+\mathbf{B}\vec\tau\:\big) +\mathbf B\vec m_n\right)_k\right\}
\end{eqnarray}
on $\ualpha_k$ by \eqref{theta-periods} and \eqref{Omega-jump}, where $\vec m_n,\vec j_n\in\Z^g$ are such that
\begin{equation}
\label{jnmn}
\vec\Omega(\mathcal{D}_n)-\vec\Omega(\mathcal{D}_*) = \vec c_\rho + n\big(\vec\omega+\mathbf{B}\vec\tau\:\big) + \vec j_n +\mathbf B\vec m_n.
\end{equation}

\subsection{Auxiliary Functions, II}
Let $\lambda_{\vec m_n}$ be the function on $\ugamma=\cup\ubeta_k$ such that $\lambda_{\vec m_n}\equiv-2\pi\mathrm{i}(\vec m_n)_k$ on $\ubeta_k$ and set
\begin{equation}
\label{smn}
S_{\vec m_n}(\z) := \exp\big\{\Lambda_{\vec m_n}(\z)\big\}, \quad \z\in\widetilde\RS.
\end{equation}
Since $\vec m_n\in \Z$, $S_{\vec m_n}$ is holomorphic across the $\ubeta$-cycles by the analytic continuation principle and therefore is holomorphic in $\widehat\RS$. It has continuous traces on the $\ualpha$-cycles that satisfy
\begin{equation}
\label{jump3}
S_{\vec m_n}^+ = S_{\vec m_n}^- \exp\big\{2\pi \mathrm{i}\big(\mathbf B\vec m_n\big )_k\big\} \quad \mbox{on} \quad \ualpha_k.
\end{equation}

As $\mathbf{B}$ has positive definite imaginary part, any vector in $\vec u\in\C^g$ can be uniquely written as $\vec x+\mathbf{B}\vec y$ for some $\vec x,\vec y\in\R^g$. Write
\[
\vec c_\rho =: \vec x_\rho + \mathbf{B}\vec y_\rho \quad \text{and} \quad \vec\Omega(\mathcal{D}_n)-\vec\Omega(\mathcal{D}_*)  =: \vec x_n+\mathbf{B}\vec y_n, \quad n\in\N,
\]
Then, of course,
\begin{equation}
\label{xnyn}
\vec x_n = \vec x_\rho + n\vec \omega + \vec j_n \quad \text{and} \quad \vec y_n = \vec y_\rho + n\vec\tau + \vec m_n
\end{equation}
by \eqref{jnmn}. Since the image of the closure of $\widetilde\RS$ under $\vec\Omega$ is bounded in $\C^g$, so are the vectors $\vec x_n,\vec y_n$. Clearly, in this case \eqref{xnyn} implies that the vectors $n\vec\omega+\vec j_n$ and $n\vec\tau+\vec m_n$ are bounded with $n$. Therefore, 
\begin{equation}
\label{SzegoBounded}
C^{-1} \leq \left|S_\rho S_{n\vec\tau+\vec m_n}\right| \leq C
\end{equation}
uniformly with $n$ in $\widetilde\RS$ for some absolute constant $C>1$.

\subsection{A Family of BVPs}
By combining all the material above, we obtain the following theorem.

\begin{theorem}
\label{thm:Psi_n}
For $n\in\N$, let $\tilde n$ be the greatest integer in $\N_*$ smaller or equal to $n$ ($\tilde n=n-k$ using the notation from \eqref{Dn}). With \eqref{phi}, \eqref{stau}, \eqref{Sp},  \eqref{thetan}, and \eqref{smn} at hand, we deduce that the function
\begin{equation}
\label{Psi}
\Psi_n:=\Phi^{\tilde n}S_\rho S_{\tilde n\vec\tau+\vec m_{\tilde n}}\Theta_{\tilde n}
\end{equation}
is sectionally meromorphic in $\RS\setminus\bd$ whose zeros and poles there\footnote{$\Psi_n$ is non-vanishing and finite in $D^{(0)}\cup D^{(1)}$ except at the elements of its divisor that stand for zeros (resp. poles) if preceded by the plus (resp. minus) sign and the integer coefficients in front of them indicate multiplicity.} are described by the divisor
\[
(n-g)\infty^{(1)}+\mathcal{D}_n-n\infty^{(0)}
\]
since $\mathcal{D}_n=\mathcal{D}_{\tilde n}+k\big(\infty^{(0)}-\infty^{(1)}\big)$. Moreover, it has continuous traces on $\bd\setminus \boldsymbol E_\Delta$ that satisfy
\[
\Psi_n^+ = \big(1/\rho\circ\pi\big)\Psi_n^-
\]
by \eqref{jump1}, \eqref{jump2}, and \eqref{jump3}, and it is bounded near the points in $\boldsymbol E_\Delta$.
\end{theorem}

To describe the asymptotic properties of $\Psi_n$ we need to further restrict $\N_*$.
\begin{definition}
\label{Nvarepsilon}
Given $\varepsilon>0$, we say that an index $n$ belongs to $\N_\varepsilon$ if and only if
\[
\pi\left(\RS^{(1)}\cap\mathcal{D}_{n-1}\right),\pi\left(\RS^{(0)}\cap\mathcal{D}_n\right)\subset \left\{z:|z|\leq\varepsilon^{-1}\right\},
\]
where we consider a divisor as a subset of points on $\RS$.
\end{definition}

The indices excluded from $\N_*$ are exactly the ones corresponding to the non-unique solutions of \eqref{main-jip}, that is, the solutions for which $\mathcal{D}_n$ contains at least one pair of $\infty^{(0)}+\infty^{(1)}$. Hence, if $n\in\N_\varepsilon$, then $n,n-1\in \N_*$. Furthermore, the subsequences $\N_\varepsilon$ are infinite for all $\varepsilon$ small enough as follows from the considerations in Section~\ref{ssec:limitpoints} (this is the precise reason why this subsection is included).

It follows immediately from the definition of $\N_\varepsilon$ that the following constants are well defined:
\begin{equation}
\label{gammans}
1/\gamma_n := \lim_{\z\to\infty^{(0)}}\Psi_n\big(\z\big)z^{-n}  \quad \mbox{and} \quad 1/\gamma_n^* := \lim_{\z\to\infty^{(1)}}\Psi_{n-1}(\z)z^{n-1-g}, \quad n\in\N_\varepsilon.
\end{equation}

\begin{lemma}
\label{lem:estimates}
For each bounded $K\subset D^{(1)}$, there exists constant  $C(K)>1$ such that
\begin{equation}
\label{maxPsi}
\max_K \left|\Psi_n\right| \leq C(K)^{-n}.
\end{equation}
Moreover, for a given $\varepsilon>0$ there exists a constant $C(\varepsilon)>1$ such that
\begin{equation}
\label{productgammas}
C(\varepsilon)^{-1} \leq \big|\gamma_n\gamma_n^*\big| \leq C(\varepsilon), \quad n\in\N_\varepsilon.
\end{equation}
\end{lemma}
\begin{proof}
To show \eqref{maxPsi}, write
\[
|\Psi_n|=\big|\Phi^{\tilde n- g}\big|\cdot\big|S_\rho S_{n\vec\tau+\vec{m_n}}\big|\cdot\big|\Phi^{g}\Theta_n\big|.
\]
The first multiple in the decomposition above is locally uniformly geometrically small in $D^{(1)}$ by \eqref{ineq1} and the second one is uniformly bounded by \eqref{SzegoBounded}. Thus, it is enough to show that the functions $\big|\Phi^g\Theta_n\big|$ are uniformly bounded in $D^{(1)}$. It is, in fact, a family of continuous function in $\overline{D^{(1)}}\setminus\bigcup\ualpha_k$ with uniformly bounded jumps on the $\ualpha$-cycles (boundedness of the jumps follows from \eqref{jump2} and the uniform boundedness of the vectors $n\vec\tau+\vec m_n$ concluded after \eqref{xnyn}). Hence, each function is bounded in $\overline{D^{(1)}}$. As the family is indexed by the divisors $\mathcal{D}_n$ that belong to $\RS^g/\Sigma_g$ and the latter space is compact, the uniform boundedness follows.

It follows again from \eqref{jump2} and \eqref{SzegoBounded} that to show \eqref{productgammas} it is sufficient to establish the uniform boundedness with $n\in\N_\varepsilon$ the absolute values of
\begin{equation}
\label{product-thetas}
\Theta_n\big(\infty^{(0)}\big)\lim_{\z\to\infty^{(1)}}\Theta_{n-1}(\z)z^{-g}.
\end{equation}
To this end, denote by $\mathfrak{C}_\varepsilon^0$ and $\mathfrak{C}_\varepsilon^1$ the closures of $\big\{\mathcal{D}_n\big\}_{n\in\N_\varepsilon}$ and $\big\{\mathcal{D}_{n-1}\big\}_{n\in\N_\varepsilon}$ in the $\RS^g/\Sigma_g$-topology. Neither of these sets contains special divisors. Indeed, both sequences consist of non-special divisors and therefore we need to consider only the limiting ones. The limit points belonging to $\mathfrak{C}_\varepsilon^0$ are necessarily of the form
\[
\mathcal{D}+\sum_{i=1}^k\left(z_i^{(0)}+z_i^{(1)}\right) + l\infty^{(1)},
\]
where the integral divisor $\mathcal{D}$ has degree $g-2k-l$, is non-special, and does not contain neither $\infty^{(0)}$ nor $\infty^{(1)}$. If $k$ were strictly positive, the considerations of Section~\ref{ssec:limitpoints} would imply that $\mathfrak{C}_\varepsilon^1$ should contain divisors of the form
\[
\mathcal{D}+\sum_{i=1}^{k^\prime}\left(w_i^{(0)}+w_i^{(1)}\right) + (k-k^\prime-1)\infty^{(0)} + (l+1+k-k^\prime)\infty^{(1)}
\]
$0\leq k^\prime\leq k-1$. In particular, it would be true that $l+1+k-k^\prime\geq2$, which is impossible by the very definition of $\N_\varepsilon$. Since the set $\mathfrak{C}_\varepsilon^1$ can be examined similarly, the claim follows. Thus, using \eqref{thetan}, we can establish a quantity similar to \eqref{product-thetas}, for the pairs of limit points in $\mathfrak{C}_\varepsilon^0\times\mathfrak{C}_\varepsilon^1$. Moreover, all these quantities are finite and non-zero as all the divisors are non-special. The claim now follows from the compactness argument.
\end{proof}

\section{Main Results}
\label{sec:RH}

Fix an algebraic S-contour $\Delta=E_0\cup E_1\cup\bigcup\Delta_j$, see Definition~\ref{SContour}, and let $w_\Delta$ be defined by \eqref{w_Delta}, $z^{-g-1}w_\Delta(z)\to1$ as $z\to\infty$. Let $\rho$ be a function holomorphic and non-vanishing in a neighborhood of each connected component of $\Delta$ (in general, $\rho$ is piecewise holomorphic). Recall \eqref{f_rho} that we set
\[
f_\rho(z) = \frac1{2\pi\mathrm{i}}\int_\Delta\frac{(\rho/w_\Delta^+)(t)}{t-z}\mathrm{d}t, \quad z\in \overline\C\setminus\Delta.
\]
Further, let $\Psi_n$, which depends on $\rho$, be defined by \eqref{Psi}. With a slight abuse of notation, put
\begin{equation}
\label{Psistar}
\Psi_n(z):=\Psi_n\big(z^{(0)}\big) \quad \text{and} \quad \Psi_n^*(z):=\Psi_n\big(z^{(1)}\big), \quad z\in D.
\end{equation}
Then it follows from Theorem~\ref{thm:Psi_n} that these functions are holomorphic in $\C\setminus\Delta$. Moreover, when $n\in\N_\varepsilon$, see Definition~\ref{Nvarepsilon}, it holds that $\Psi_n$ has a pole of exact order $n$ and at infinity, $\Psi_{n-1}$ has a pole of order at most $n-1$ there, $\Psi_n^*$ vanishes at infinity, and $\Psi_{n-1}^*$ has a zero of exact order $n-1-g$ there. Furthermore, it holds that
\begin{equation}
\label{Psistarjumps}
\big(\Psi_n^*\big)^\pm = \rho \Psi_n^\mp \quad \text{on} \quad \bigcup\Delta_j,
\end{equation}
where all the traces are continuous on $\bigcup\Delta_j$ and are bounded near $e\in E_0\cup E_1$. Finally, let $\gamma_n$ and $\gamma_n^*$ be defined by \eqref{gammans}. Then the following theorem holds.

\begin{theorem}
\label{thm:main}
Let $[n/n]_{f_\rho}=P_n/Q_n$ be the $n$-th diagonal Pad\'e approximant to $f_\rho$ defined by \eqref{f_rho} with $\rho$ holomorphic and non-vanishing on $\Delta$ and $R_n$ be the linearized error of approximation given by \eqref{linsys}. Then for all $n\in\N_\varepsilon$ large enough it holds that
\begin{equation}
\label{SA1}
\left\{
\begin{array}{rll}
Q_n &=& \displaystyle \left(1+\upsilon_{n1}\right) \gamma_n\Psi_n + \upsilon_{n2}\gamma_n^* \Psi_{n-1},\bigskip \\
w_\Delta R_n &=& \displaystyle \left(1+\upsilon_{n1}\right)\gamma_n \Psi_n^*+\upsilon_{n2}\gamma_n^*\Psi_{n-1}^*,
\end{array}
\right.
\end{equation}
locally uniformly in $\overline\C\setminus\Delta$, where $\upsilon_{nj}(\infty)=0$ and $|\upsilon_{nj}|\leq C_\varepsilon^{-n}$ in $\overline\C$ for some constant $C_\varepsilon>1$.
\end{theorem}

In the case where $g>0$, formulae \eqref{SA1} clearly indicate the absence of uniform convergence of Pad\'e approximants. Indeed, the error of approximation is equal to
\[
f_\rho - [n/n]_{f_\rho} = \frac{R_n}{Q_n} = \frac1{w_\Delta} \frac{\Psi_n^*}{\Psi_n}\frac{1+\upsilon_{n1}+\upsilon_{n2}(\gamma_n^*/\gamma_n)\big(\Psi_{n-1}^*/\Psi_n^*\big)}{1+\upsilon_{n1}+\upsilon_{n2}(\gamma_n^*/\gamma_n)\big(\Psi_{n-1}/\Psi_n\big)}.
\]
We do know from Lemma \ref{lem:estimates} that the functions $\Psi_n^*$ are geometrically small on closed subsets of $D$. Similar argument can be used to show that the functions $\Psi_n$ are geometrically large in $D$ except for possible zeros described by those elements of the divisor $\mathcal{D}_n$ that belong to $D^{(0)}$ (Rouch\'e's theorem clearly implies that $Q_n$ has a zero close to the canonical projection of each such element) and those zeros are the sole reason why the uniform convergence does not hold. In the ``generic case'', i.e., when $1$ and the periods \eqref{constants} of the Green differential $\mathrm{d}G$ are rationally independent, it is known \cite[pages 190--191]{Wid69} that the divisors $\mathcal{D}_n$ are dense in $\RS^g/\Sigma_g$ and hence one will definitely observe the presence of wandering poles. However, in this generic case, there always exists a subsequence of indices such that the elements of the divisors $\mathcal{D}_n$ belong only to $D^{(1)}$ \cite[Sec.~4.1]{Suet00} and therefore there always exists a subsequence along which Pad\'e approximants $[n/n]_{f_\rho}$ converge to $f_\rho$ locally uniformly in $D$.

The remaining part of this section is devoted to the proof of Theorem~\ref{thm:main}.

\subsection{Initial R-H Problem}

Below, we follow by now classical approach of Fokas, Its, and Kitaev \cite{FIK91,FIK92} connecting orthogonal polynomials to matrix Riemann-Hilbert problems. To this end, assume that the index $n$ is such that
\begin{equation}
\label{assumption}
\deg(Q_n)=n \quad \text{and} \quad R_{n-1}(z)\sim z^{-n} \quad \text{as} \quad z\to\infty.
\end{equation}
Define
\begin{equation}
\label{eq:y}
{\boldsymbol Y} = \left(\begin{array}{cc}
Q_n & R_n \\
m_{n-1}Q_{n-1} & m_{n-1}R_{n-1}
\end{array}\right),
\end{equation}
where $m_n$ is a constant such that $m_{n-1}R_{n-1}(z)=z^{-n}[1+o(1)]$ near infinity. Then $\boldsymbol{Y}$ solves the following matrix Riemann-Hilbert problem (\rhy) :
\begin{itemize}
\label{rhy}
\item[(a)] ${\boldsymbol Y}$ is analytic in $\overline\C\setminus\Delta$ and $\displaystyle \lim_{z\to\infty} {\boldsymbol Y}(z)z^{-n\sigma_3} = {\boldsymbol I}$, where ${\boldsymbol I} = \left(\begin{array}{cc} 1&0\\0& 1 \end{array}\right)$ and $\displaystyle \sigma_3 = \left(\begin{array}{cc} 1&0\\0&-1 \end{array}\right)$;
\item[(b)] ${\boldsymbol Y}$ has continuous traces on $\bigcup\Delta_j$ that satisfy $\displaystyle {\boldsymbol Y}_+ = {\boldsymbol Y}_- \left(\begin{array}{cc}1&\rho/w_\Delta^+\\0&1\end{array}\right);$
\item[(c)] ${\boldsymbol Y}$ is bounded near each $e\in (E_0\cup E_1)\setminus E_\Delta$ and the behavior of ${\boldsymbol Y}$ near each $e\in E_\Delta$ is described by $\displaystyle O\left(\begin{array}{cc}1&|z-e|^{-1/2}\\1&|z-e|^{-1/2}\end{array}\right)$ as $D\ni z\to e$.
\end{itemize}

The property \hyperref[rhy]{\rhy}(a) follows immediately from \eqref{linsys} and \eqref{assumption}. The property \hyperref[rhy]{\rhy}(b) is due to the equality
\[
R_n^+ - R_n^- = Q_n\left(f_\rho^+ - f_\rho^-\right) = Q_n\rho/w_\Delta^+ \quad \text{on} \quad \bigcup\Delta_j,
\]
which in itself is a consequence of \eqref{linsys}, \eqref{f_rho}, and the Sokhotski-Plemelj formulae \cite[Section~4.2]{Gakhov}. Finally, to show \hyperref[rhy]{\rhy}(c), write, $R_n=\sum_k R_{nk}$, where
\[
R_{nk}(z) := \frac1{2\pi \mathrm{i}}\int_{\Delta_k}\frac{(Q_n\rho/w_\Delta^+)(t)}{t-z}\mathrm{d}t
\]
and therefore the behavior of $R_n$ near $e\in E_0\cup E_1$ is deduced from the behavior $R_{nk}$ there. If the endpoint $e$ of $\Delta_k$ has an odd bifurcation index ($e\in E_\Delta$), then $w_\Delta^2$ has a simple zero there and therefore $|R_{nk}(z)|\sim|z-e|^{-1/2}$ as $z\to e$, see \cite[Section~3]{BY13}. On the other hand, if $e$ has an even bifurcation index ($e\in (E_0\cup E_1)\setminus E_\Delta$), the respective function $R_{nk}$ behaves as 
\[
\frac{\rho(e)w^+_{\Delta|\Delta_k}(e)}{2\pi\mathrm{i}}\log(z-e)+R_{e,k}^*(z)
\]
according to \cite[Section~8.1]{Gakhov}, where the function $R_{e,k}^*$ has a definite limit at $e$ and the logarithm is holomorphic outside of $\Delta_k$. Since $w_\Delta$ does not have a branch point at $e$, it holds that $\sum_kw^+_{\Delta|\Delta_k}(e)=0$, where the sum is taken over all arcs $\Delta_k$ incident with $e$. Thus, we get that
\[
R(z) = \frac{\rho(e)}{2\pi}\sum_k\arg_{e,k}(z-e) + R_e^*(z),
\]
where $R_e^*$ has a definite limit at $e$, $\arg_{e,k}(z-e)$ has the branch cut along $\Delta_k$, and the sum is again taken over all arcs incident with $e$. Thus, $\boldsymbol{Y}$ is bounded in the vicinity of each $e$ with even bifurcation index.

To show that a solution of \hyperref[rhy]{\rhy}, if exists, must be of the form \eqref{eq:y} is by now a standard exercise, see for instance, \cite[Lemma~2.3]{KMcLVAV04}, \cite{BY10}, \cite[Lemma~1]{uAY}. Thus, we proved the following lemma.

\begin{lemma}
\label{lem:rhy}
If a solution of \hyperref[rhy]{\rhy} exists then it is unique. Moreover, in this case it is given by \eqref{eq:y} where $Q_n$ and $R_{n-1}$ satisfy \eqref{assumption}. Conversely, if \eqref{assumption} is fulfilled, then \eqref{eq:y} solves \hyperref[rhy]{\rhy}.
\end{lemma}

\subsection{Transformed R-H Problem}

It can be directly verified that
\[
\left(\begin{array}{cc} 1 & 0 \\ -w_\Delta^-/\rho & 1 \end{array}\right) \left(\begin{array}{cc} 0 & \rho/w_\Delta^+ \\ -w_\Delta^+/\rho & 0 \end{array}\right) \left(\begin{array}{cc} 1 & 0 \\ w_\Delta^+/\rho & 1 \end{array}\right) = \left(\begin{array}{cc} 1 & \rho/w_\Delta^+ \\ 0 & 1 \end{array}\right).
\]
This factorization of the jump matrix in \hyperref[rhy]{\rhy}(b) suggests the following transformation of $\boldsymbol{Y}$:
\begin{equation}
\label{eq:x}
{\boldsymbol X}:= \left\{
\begin{array}{ll}
{\boldsymbol Y} \left(\begin{array}{cc}1&0 \\ -w_\Delta/\rho & 1 \end{array}\right), & \mbox{in} \quad \Omega, \smallskip \\
{\boldsymbol Y}, & \mbox{in} \ \C\setminus\overline\Omega,
\end{array}
\right.
\end{equation}
where $\Omega$ is an open set bounded by $\Delta$ and $\Gamma$ and $\Gamma$ is a union of simple Jordan curves each encompassing one connected component of $\Delta$ and chosen so $\rho$ is holomorphic across $\Gamma$. It is trivial to verify that ${\boldsymbol X}$ solves the following Riemann-Hilbert problem (\rhs):
\begin{itemize}
\label{rhs}
\item[(a)] ${\boldsymbol X}$ is analytic in $\C\setminus(\Delta\cup\Gamma)$ and $\displaystyle \lim_{z\to\infty} {\boldsymbol X}(z)z^{-n\sigma_3} = {\boldsymbol I}$;
\item[(b)] ${\boldsymbol X}$ has continuous traces on $\bigcup\Delta_j\cup\Gamma$ that satisfy
\[
{\boldsymbol X}_+={\boldsymbol X}_- \left\{
\begin{array}{rl}
\displaystyle \left(\begin{array}{cc} 0 & \rho/w_\Delta^+ \\ -w_\Delta^+/\rho & 0 \end{array}\right) & \text{on} \quad \bigcup\Delta_j  \smallskip \\
\displaystyle \left(\begin{array}{cc} 1 & 0 \\ w_\Delta/\rho & 1 \end{array}\right) & \text{on} \quad \Gamma;
\end{array}
\right.
\]
\item[(c)] ${\boldsymbol X}$ has the behavior near $e\in E_0\cup E_1$ described by \hyperref[rhy]{\rhy}(c).
\end{itemize}

Then the following lemma can be easily checked.

\begin{lemma}
\label{lem:rhs}
\hyperref[rhs]{\rhs} is solvable if and only if \hyperref[rhy]{\rhy} is solvable. When solutions of \hyperref[rhs]{\rhs} and \hyperref[rhy]{\rhy} exist, they are unique and connected by \eqref{eq:x}.
\end{lemma}

\subsection{Asymptotics in the Bulk}
\label{sec:bulk}

Let $\Psi_n,\Psi_n^*$ be defined by \eqref{Psistar} and $\gamma_n,\gamma_n^*$ be as in \eqref{gammans}. Set
\begin{equation}
\label{eq:n}
{\boldsymbol N} := \left(\begin{array}{cc} \gamma_n & 0 \smallskip \\ 0 &  \gamma_n^* \end{array}\right) \widetilde{\boldsymbol N}, \quad \widetilde{\boldsymbol N}:=\left(\begin{array}{cc} \Psi_n & \Psi_n^*/w_\Delta \smallskip \\ \Psi_{n-1} &  \Psi_{n-1}^*/w_\Delta \end{array}\right).
\end{equation}
Then $\boldsymbol{N}$ solves the following Riemann-Hilbert problem (\rhn):
\begin{itemize}
\label{rhn}
\item[(a)] ${\boldsymbol N}$ is analytic in $\C\setminus\Delta$ and $\displaystyle \lim_{z\to\infty} {\boldsymbol N}(z)z^{-n\sigma_3} = {\boldsymbol I}$;
\item[(b)] ${\boldsymbol N}$ has continuous traces on $\bigcup\Delta_j$ that satisfy $\displaystyle {\boldsymbol N}_+={\boldsymbol N}_-\left(\begin{array}{cc} 0 & \rho/w_\Delta^+ \\ -w_\Delta^+/\rho & 0 \end{array}\right)$;
\item[(c)] ${\boldsymbol N}$ has the behavior near $e\in E_0\cup E_1$ described by \hyperref[rhy]{\rhy}(c).
\end{itemize}

Indeed, \hyperref[rhn]{\rhn}(a) follows immediately from the analyticity properties of $\Psi_n,\Psi_n^*$ and \eqref{gammans}. \hyperref[rhn]{\rhn}(b) can be easily checked by using \eqref{Psistarjumps}. Finally, \hyperref[rhn]{\rhn}(c) is the consequences of the boundedness of $\Psi_n^\pm$ and $(\Psi_n^*)^\pm$ on $\bigcup\Delta_j$ and the definition of $w_\Delta$.

Moreover, it can be readily checked that $\det(\boldsymbol{N})$ is a holomorphic function in $\overline\C\setminus (E_0\cup E_1)$ and $\det(\boldsymbol{N})(\infty)=1$. Since it is either bounded or behaves like $O\left(|z-e|^{-1/2}\right)$ near $e\in E_0\cup E_1$, those points are in fact removable singularities and therefore $\det(\boldsymbol{N})$ is a bounded entire function. That is, $\det(\boldsymbol{N})\equiv1$ as follows from the normalization at infinity.

\subsection{Final R-H Problem}
Consider the following Riemann-Hilbert Problem (\rhr):
\begin{itemize}
\label{rhr}
\item[(a)] $\boldsymbol{Z}$ is a holomorphic matrix function in $\overline\C\setminus\Gamma$ and $\boldsymbol{Z}(\infty)=\boldsymbol{I}$;
\item[(b)] $\boldsymbol{Z}$ has continuous traces on $\Gamma$ that satisfy $\displaystyle \boldsymbol{Z}_+  = \boldsymbol{Z}_- \widetilde{\boldsymbol N} \left(\begin{array}{cc} 1 & 0 \\ w_\Delta/\rho & 1 \end{array}\right) \widetilde{\boldsymbol N}^{-1}$.
\end{itemize}
Then the following lemma takes place.
\begin{lemma}
\label{lem:rhr}
The solution of \hyperref[rhr]{\rhr} exists for all $n\in\N_\varepsilon$ large enough and satisfies
\begin{equation}
\label{eq:r}
\boldsymbol{Z}=\boldsymbol{I}+O\big( C_\varepsilon^{-n}\big)
\end{equation}
for some constant $C_\varepsilon>1$ independent of $\Gamma$, where $O(\cdot)$ holds uniformly in $\overline\C$.
\end{lemma}
\begin{proof}
Since $\det(\boldsymbol{N})\equiv1$ and therefore $\det(\widetilde{\boldsymbol{N}})\equiv1/(\gamma_n\gamma_n^*)$, the jump matrix for $\boldsymbol{Z}$ is equal to
\[
\boldsymbol{I} + \frac{\gamma_n\gamma_n^*}{\rho w_\Delta} \left(\begin{array}{cc} \Psi_n^*\Psi_{n-1}^* & -\big(\Psi_n^*\big)^2 \smallskip \\ \big(\Psi_{n-1}^*\big)^2 &  -\Psi_n^*\Psi_{n-1}^* \end{array}\right) = \boldsymbol{I} + O\big( C_{\varepsilon,\Gamma}^{-2n}\big),
\]
where the last equality follows from Lemma~\ref{lem:estimates}. Therefore, according to \cite[Corollary~7.108]{Deift}, \hyperref[rhr]{\rhr} is solvable for all $n\in\N_\varepsilon$ large enough and $\boldsymbol{Z}_\pm$ converge to zero on $\Gamma$ in $L^2$-sense geometrically fast. The latter yields \eqref{eq:r} locally uniformly in $\overline\C\setminus\Gamma$ with some constant $C_{\varepsilon,\Gamma}^*>1$. Consider $\widetilde\Gamma$ homotopic to and disjoint from $\Gamma$ which also lies within the domain of analyticity of $\rho$. The above considerations yield a solution $\widetilde{\boldsymbol{Z}}$ of  \hyperref[rhr]{\rhr}~ with the jump matrix defined on $\widetilde\Gamma$ rather than $\Gamma$. As the jump matrices for $\widetilde{\boldsymbol{Z}}$ and $\boldsymbol{Z}$ are analytic continuations of each other, so are the solutions $\widetilde{\boldsymbol{Z}}$ and $\boldsymbol{Z}$. Hence, \eqref{eq:r} indeed holds with $C_\varepsilon:=\min\left\{C_{\varepsilon,\Gamma}^*,C_{\varepsilon,\widetilde\Gamma}^*\right\}$.
\end{proof}

\subsection{Asymptotics}
Let $\boldsymbol{Z}$ be a solution of \hyperref[rhr]{\rhr} granted by Lemma~\ref{lem:rhr} and $\widetilde{\boldsymbol{N}}$ be the matrix function constructed in \eqref{eq:n}. Then it can be easily checked that
\[
\boldsymbol{X} = \left(\begin{array}{cc} \gamma_n & 0 \smallskip \\ 0 &  \gamma_n^* \end{array}\right)\boldsymbol{Z}\widetilde{\boldsymbol{N}}
\]
solves \hyperref[rhs]{\rhs} and therefore
\[
{\boldsymbol Y}:= \left(\begin{array}{cc} \gamma_n & 0 \smallskip \\ 0 &  \gamma_n^* \end{array}\right) \boldsymbol{ZN} \left\{
\begin{array}{ll}
 \left(\begin{array}{cc}1&0 \\ w_\Delta/\rho & 1 \end{array}\right), & \mbox{in} \quad \Omega, \smallskip \\
\boldsymbol{I}, & \mbox{in} \ \C\setminus\overline\Omega,
\end{array}
\right.
\]
solves \hyperref[rhy]{\rhy} by Lemma~\ref{lem:rhs}. Given any closed set $K\subset\overline\C\setminus\Delta$, choose $\Omega$ so that $K\subset \overline\C\setminus\overline\Omega$.  Write
\[
\boldsymbol{Z} = \left(\begin{array}{cc} 1+\upsilon_{n1} & \upsilon_{n2} \\ \upsilon_{n3} & 1+\upsilon_{n4} \end{array}\right),
\]
where we know from Lemma~\ref{lem:rhr} that $|\upsilon_{nk}|\leq C_\varepsilon^{-n}$ uniformly in $\overline\C$ ($\upsilon_{nk}(\infty)=0$ as $\boldsymbol{Z}(\infty)=\boldsymbol{I}$). Then
\[
[\boldsymbol{Y}]_{1i} = \big(1+\upsilon_{n1}\big)[\boldsymbol{N}]_{1i} + \upsilon_{n2}[\boldsymbol{N}]_{2i}, \quad i\in\{1,2\}.
\]
The claim of Theorem~\ref{thm:main} now follows from \eqref{eq:y} and \eqref{eq:n}.

\singlespacing

\bibliographystyle{plain}
\bibliography{../../bibliography}

\begin{thebibliography}{10}

\bibitem{Akh60}
N.I. Akhiezer.
\newblock Orthogonal polynomials on several intervals.
\newblock {\em Dokl. Akad. Nauk {SSSR}}, 134:9--12, 1960.
\newblock English transl. in {\it{S}oviet {M}ath. {D}okl.} 1, 1960.

\bibitem{Akhiezer}
N.I. Akhiezer.
\newblock {\em Elements of the Theory of Elliptic Functions}.
\newblock Amer. Math. Soc., Providence, RI, 1990.

\bibitem{ApVA04}
A.I. Aptekarev and W.~Van Assche.
\newblock Scalar and matrix {R}iemann-{H}ilbert approach to the strong
  asymptotics of {P}ad\'e approximants and complex orthogonal polynomials with
  varying weight.
\newblock {\em J. Approx. Theory}, 129:129--166, 2004.

\bibitem{uAY}
A.I. Aptekarev and M.~Yattselev.
\newblock Pad\' e approximants for functions with branch points --- strong
  asymptotics of {N}uttall-{S}tahl polynomials.
\newblock Submitted for publication. \url{http://arxiv.org/abs/1109.0332}.

\bibitem{BakerGravesMorris}
G.A. Baker and P.~Graves-Morris.
\newblock {\em Pad\'e Approximants}, volume~59 of {\em Encyclopedia of
  Mathematics and its Applications}.
\newblock Cambridge University Press, 1996.

\bibitem{BY10}
L.~Baratchart and M.~Yattselev.
\newblock Convergent interpolation to {C}auchy integrals over analytic arcs
  with {J}acobi-type weights.
\newblock {\em Int. Math. Res. Not.}, 2010.
\newblock Art. ID rnq 026, pp. 65.

\bibitem{BY13}
L.~Baratchart and M.~Yattselev.
\newblock {P}ad\'e approximants to a certain elliptic-type functions.
\newblock {\em J. Anal. Math.}, 121:31--86, 2013.

\bibitem{Deift}
P.~Deift.
\newblock {\em Orthogonal Polynomials and Random Matrices: a Riemann-Hilbert
  Approach}, volume~3 of {\em Courant Lectures in Mathematics}.
\newblock Amer. Math. Soc., Providence, RI, 2000.

\bibitem{M-FRakhSuet12}
A.~Mart\'inez Finkelshtein, E.A. Rakhmanov, and S.P. Suetin.
\newblock {H}eine, {H}ilbert, {P}ad\'e, {R}iemann, and {S}tieljes: a {J}ohn
  {N}uttall's work 25 years later.
\newblock In J.~Arves\'u and G.~L\'opez Lagomasino, editors, {\em Recent
  Advances in Orthogonal Polynomials, Special Functions, and Their
  Applications}, volume 578, pages 165---193, 2012.
\newblock \url{http://arxiv.org/abs/1111.6139}.

\bibitem{FIK91}
A.S. Fokas, A.R. Its, and A.V. Kitaev.
\newblock Discrete {P}anlev\'e equations and their appearance in quantum
  gravity.
\newblock {\em Comm. Math. Phys.}, 142(2):313--344, 1991.

\bibitem{FIK92}
A.S. Fokas, A.R. Its, and A.V. Kitaev.
\newblock The isomonodromy approach to matrix models in {2D} quantum
  gravitation.
\newblock {\em Comm. Math. Phys.}, 147(2):395--430, 1992.

\bibitem{Gakhov}
F.D. Gakhov.
\newblock {\em Boundary Value Problems}.
\newblock Dover Publications, Inc., New York, 1990.

\bibitem{Gon82}
A.A. Gonchar.
\newblock On uniform convergence of diagonal {P}ad\'e approximants.
\newblock {\em Math. USSR Sb.}, 43(527--546), 1982.

\bibitem{KMcLVAV04}
A.B. Kuijlaars, K.T.-R. McLaughlin, W.~Van Assche, and M.~Vanlessen.
\newblock The {R}iemann-{H}ilbert approach to strong asymptotics for orthogonal
  polynomials on $[-1,1]$.
\newblock {\em Adv. Math.}, 188(2):337--398, 2004.

\bibitem{Nut84}
J.~Nuttall.
\newblock Asymptotics of diagonal {H}ermite-{P}ad\'e polynomials.
\newblock {\em J. Approx. Theory}, 42(4):299--386, 1984.

\bibitem{Nut90}
J.~Nuttall.
\newblock Pad\'e polynomial asymptotic from a singular integral equation.
\newblock {\em Constr. Approx.}, 6(2):157--166, 1990.

\bibitem{NutS77}
J.~Nuttall and S.R. Singh.
\newblock Orthogonal polynomials and {P}ad\'e approximants associated with a
  system of arcs.
\newblock {\em J. Approx. Theory}, 21:1--42, 1977.

\bibitem{Pade92}
H.~Pad\'e.
\newblock Sur la repr\'esentation approch\'ee d'une fonction par des fractions
  rationnelles.
\newblock {\em Ann. Sci Ecole Norm. Sup.}, 9(3):3--93, 1892.

\bibitem{uPerevRakh}
E.A. Perevoznikova and E.A. Rakhmanov.
\newblock Variation of the equilibrium energy and {S}-property of compacta of
  minimal capacity.
\newblock Manuscript, 1994.

\bibitem{Ransford}
T.~Ransford.
\newblock {\em Potential Theory in the Complex Plane}, volume~28 of {\em London
  Mathematical Society Student Texts}.
\newblock Cambridge University Press, Cambridge, 1995.

\bibitem{St85}
H.~Stahl.
\newblock Extremal domains associated with an analytic function. {I, II}.
\newblock {\em Complex Variables Theory Appl.}, 4:311--324, 325--338, 1985.

\bibitem{St85b}
H.~Stahl.
\newblock Structure of extremal domains associated with an analytic function.
\newblock {\em Complex Variables Theory Appl.}, 4:339--356, 1985.

\bibitem{St86}
H.~Stahl.
\newblock Orthogonal polynomials with complex valued weight function. {I, II}.
\newblock {\em Constr. Approx.}, 2(3):225--240, 241--251, 1986.

\bibitem{St97}
H.~Stahl.
\newblock The convergence of {P}ad\'e approximants to functions with branch
  points.
\newblock {\em J. Approx. Theory}, 91:139--204, 1997.

\bibitem{St98}
H.~Stahl.
\newblock Spurious poles in {P}ad\'e approximation.
\newblock {\em J. Comput. Appl. Math.}, 99:511--527, 1998.

\bibitem{Suet00}
S.P. Suetin.
\newblock Uniform convergence of {P}ad\'e diagonal approximants for
  hyperelliptic functions.
\newblock {\em Mat. Sb.}, 191(9):81--114, 2000.
\newblock English transl. in {\it {M}ath. {S}b.} 191(9):1339--1373, 2000.

\bibitem{Suet03}
S.P. Suetin.
\newblock Convergence of {C}hebysh\"ev continued fractions for elliptic
  functions.
\newblock {\em Mat. Sb.}, 194(12):63--92, 2003.
\newblock English transl. in {\it {M}ath. {S}b.} 194(12):1807--1835, 2003.

\bibitem{Wid69}
H.~Widom.
\newblock Extremal polynomials associated with a system of curves in the
  complex plane.
\newblock {\em Adv. Math.}, 3:127--232, 1969.

\bibitem{uY3}
M.~Yattselev.
\newblock Nuttall's theorem on algebraic {S}-contours.
\newblock \emph{To be submitted.}

\bibitem{Zver71}
E.I. Zverovich.
\newblock Boundary value problems in the theory of analytic functions in
  {H}\"older classes on {R}iemann surfaces.
\newblock {\em Russian Math. Surveys}, 26(1):117--192, 1971.

\end{thebibliography}

\end{document}